\documentclass[a4paper,11pt,intlimits,oneside]{amsart}
\usepackage{showkeys}
\usepackage{dsfont}

\usepackage{enumerate}
\usepackage{amsfonts}
\usepackage{amsmath,amsthm,amssymb}

\newcommand{\comment}[1]{}

\newtheorem{theorem}{Theorem}
\newtheorem{definition}[theorem]{Definition}
\newtheorem{proposition}[theorem]{Proposition}
\newtheorem{lemma}[theorem]{Lemma}

\newtheorem{remark}[theorem]{Remark}

\newcommand\supp{\qopname\relax o{supp}}

\newcommand{\NN}{\mathbb N}
\newcommand{\ZZ}{\mathbb Z}

\newcounter{rea}
\begin{document}

\title[Uncertainty principle]{Equality cases for the uncertainty principle in finite Abelian groups }

\author{Aline Bonami \& SaifAllah Ghobber}


\address[Aline Bonami]{
\newline\indent F\'ed\'eration Denis Poisson
\newline\indent MAPM-UMR 6628 CNRS
\newline\indent Universit\'e d'Orl\'eans
\newline\indent 45067 Orl\'eans France.}
\email{aline.bonami@univ-orleans.fr}
\address[SaifAllah Ghobber]{
\newline\indent D\'{e}partement de      Math\'{e}matiques Appliqu\'ees
\newline\indent Institut   Pr\'eparatoire  aux       \'{e}tudes    d'ing\'{e}nieurs  de Nabeul
\newline\indent Universit\'e de  Carthage
\newline\indent Campus Universitaire, Merazka, 8000, Nabeul, Tunisie.}
\email{Saifallah.Ghobber@math.cnrs.fr}

\subjclass{42A99}
\keywords{Uncertainty Principle, Finite Abelian Groups, Fourier Matrices.}
\thanks{The authors have been partially supported by the project ANR AHPI number ANR-07-BLAN-0247-01 and CMCU program 07G 1501.}

\begin{abstract}
We consider the families of finite Abelian groups $\ZZ/p\ZZ\times \ZZ/p\ZZ$, $\ZZ/p^2\ZZ$ and $\ZZ/p\ZZ\times \ZZ/q\ZZ$ for $p,q$ two distinct prime numbers.
For the two first families we give a simple characterization of all functions whose  support has cardinality $k$ while the size of the spectrum satisfies a minimality condition.  We do it for a large number of values of $k$ in the third case. Such equality cases were previously known when $k$ divides the cardinality of the group, or for groups $\ZZ/p\ZZ$.
\end{abstract}

\maketitle

\section{Introduction}\label{sec:intro}

In this work we consider a finite Abelian group $G$, which can always be described as
\begin{equation}\label{group}
 G=\ZZ/{p_1}^{n_1}\ZZ \times \cdots \times \ZZ/{p_r}^{n_r}\ZZ  ,\end{equation}
  where the integers $p_i$ are prime numbers with possible repetition.

  We will write
  $$\ZZ_n:=\ZZ/n\ZZ$$
  to simplify notation.

  Uncertainty principles show  how small the support and the spectrum of a nonzero function $f$   may be simultaneously. The Fourier transform of $f$ is defined, for $\chi \in \widehat G$, as
  $$\widehat f (\chi):=\sum_{x\in G} f(x)\chi(-x).$$
  Here $\widehat G$ is the group of characters of $G$, which identifies to $G$. More precisely, for $G$ given by \eqref{group}, for some element $x$,
  which may be written as $x=(x_1, \ldots, x_{r})$, and some character $\chi$  that identifies with $y=(y_1, \ldots, y_{r})$,  \begin{equation}\label{character}
\chi(x)=\exp\left(2\pi i\sum _{j=1}^r\frac{x_jy_j}{p_j^{\,n_j}}\right).
\end{equation}
The spectrum of $f$ is the support of its Fourier transform $\widehat f$.  We refer to \cite{Terras} for background on finite Abelian groups.

  The first well-known estimate has been stated  by Matolcsi and Sz\"ucs in \cite{MS}. It is usually referred to as Donoho-Stark Uncertainty Principle and deals simultaneously with cardinalities of the supports of  a nonzero function $f$ and its Fourier transform $\widehat f$ (see \cite{DS} or \cite{Terras}):
  \begin{equation}\label{ds}
    |\supp (f)|\times |\supp (\widehat f)|\geq |G|.
  \end{equation}
 Here  $|A|$ stands for the cardinality of the finite set $A$. Let us also fix the following notation. For $A$ a subset of $G$ we note $ \mathds{1}_{A}$ its characteristic function. When $A$ is reduced to one point $a$, it is the Dirac mass at $a$. We note $\delta_G$ the Dirac mass at $0$ on the group $G$, so that  $\delta_G=\mathds{1}_{\{0\}}$.

 Equality cases for this inequality have been entirely described (see \cite{DS}), that is, nonzero functions $f$ for which $|\supp (f)|\times |\supp (\widehat f)|= |G|$. Up to translation, modulation and multiplication by a constant, they are given by characteristic functions of subgroups of $G$.

 Then, it has been observed by Tao in \cite{Tao}\footnote{This was also  observed  by other authors, in particular by A. Bir\'o, see the Arxiv paper
  of Frenkel \cite{Fr}. One  finds there different references for the lemma of Chebotarev  on subdeterminants of a Van der Monde matrix (1926, \cite{Ch}), which is the clue of the proof. Remark that Chebotarev's Lemma is easily deduced from a theorem of Mitchell (1881, \cite{Mi}), as observed in \cite{DVB1}.} that Inequality \eqref{ds} can be considerably improved for $\ZZ_p$ when $p$ is a prime number. Namely, he proved the following theorem.
 \begin{theorem}[Tao]\label{Tao}
  When $f$ is a nonzero function on $\ZZ_p$ with $p$ prime, then
 \begin{equation}\label{tao}
    |\supp (f)|+ |\supp (\widehat f)|\geq |G|+1.
  \end{equation}
Moreover, for any $A\subset G$ and $B\subset \widehat G$ such that $|A|+|B|=|G|+s$, the space of functions with support in $A$ and spectrum in $B$ is exactly of dimension $s$.
\end{theorem}
In Tao's paper  \cite{Tao}, the second part of the theorem is not exactly stated in this way, but this is seen by an easy modification of the proof.

Tao's Theorem contains a complete (but non explicit) description of equality cases, that is, of all nonzero functions $f$ for which $ |\supp (f)|+ |\supp (\widehat f)|= |G|+1$. Namely,  given $A$ and $B$ such that $|A|+|B|=|G|+1$, there is a unique (up to a constant) function $f$ such that $\supp (f)=A$ and $\supp(\widehat f)=B$.

\medskip

In order to describe the situation for any finite Abelian group, let us give some definitions. Firstly, for any nonempty set $A$ we denote $L(A)$ the space of complex functions on $A$. Then we will use the following notations.
\begin{definition}
For $k,l$ two positive integers, we set
$$E(k,l):= \Big\{f \in L(G); |\supp(f)|\leq k, |\supp(\widehat{f})|=l \Big\}.$$
$$E_0(k,l):= \Big\{f \in L(G); |\supp(f)|= k, |\supp(\widehat{f})|=l \Big\}.$$
\end{definition}

Next, for $1\leq k\leq |G|$, let us define {\sl  Meshulam's Function}, which we denote by $\theta(\cdot, G)$. It has been  introduced by Meshulam in \cite{M} as
\begin{equation}
\theta (k, G) := \min \{l ;\;  \ \ E(k,l)\neq \emptyset\}. \label{theta}
\end{equation}
For $|G|$ prime, by Tao's Theorem we have  $\theta (k,G)=|G|-k+1$ while, in other cases, we have only the inequality $\theta (k,G)\leq|G|-k+1$
(see Lemma \ref{dim1b} below), with equality when $k=|G|-1$ only (see Proposition \ref{caseminus1}).  Donoho-Stark's Uncertainty Principle asserts that in general $\theta (k,G)\geq |G|/k$, with possible equality when $k$ is a divisor of $|G|$.

Meshulam has given a better lower bound for $\theta(\cdot, G)$ in \cite{M}, see also \cite{KPR} for comments and extensions to the windowed Fourier transform. More precisely, let $u(\cdot, G)$ be the largest convex function on $[1, |G|]$ that coincides with $|G|/d$ at each divisor $d$ of $|G|$. Equivalently, $u(\cdot, G)$  is continuous and linear between two consecutive divisors of $|G|$. Then Meshulam has shown that $\theta(\cdot, G)\geq u(\cdot, G)$. This inequality is not sharp in general. We will see in particular that it is not sharp in general for groups $\ZZ_p\times \ZZ_q$, with $p$ and $q$ two different prime numbers.

The same problem has been considered recently by Delvaux and Van Barel \cite{DVB1, DVB2} with a different vocabulary. These authors give a large number of examples and revisit proofs with elementary methods of linear algebra. They give the precise value of Meshulam's Function as a minimum (while Meshulam stated only an inequality), see for instance Theorem 5 in \cite{DVB2}. They also have partial results in the direction that we consider here.

We are interested in the values $k$ for which there exist equality cases according to the following definition.
\begin{definition}
We say that there are equality cases for $(k, G)$ if the set $ E_0(k, \theta(k, G))$ is not empty. In this case, we call equality case for $(k, G)$ any nonzero function $ f \in L(G)$ that belongs to the set $ E_0(k, \theta(k, G))$. We say that $f$ is an equality case for $G$  when it is an equality case for some $(k,G)$.
\end{definition}
Delvaux and Van Barel implicitly pose the problem of finding all equality cases, that is, giving a complete description of the set $ E_0(k, \theta(k, G))$ for all $(k, G)$.

We will address this question in three particular cases. More precisely, we will consider groups $\ZZ_{p^2}$, $\ZZ_p\times \ZZ_p$ and $\ZZ_p\times \ZZ_q$, for $p,q$ distinct prime numbers. In  these three cases, except for a small number of values of $k$ in the last case, we are able to give a simple description of all equality cases, in the same spirit as  the already  known one for $k$ a divisor of $|G|$. It is particularly simple to describe the equality cases in the third case.
 \begin{theorem} \label{twoprimes} Let $p>q$ be two  prime numbers. Then, for $k$ such that $\theta(k, \ZZ_p\times \ZZ_q)<\theta(k-1, \ZZ_p\times \ZZ_q)$, except perhaps for  $k=q\frac{p+1}{q+1}$ when  $q+1$ divides $p+1$,
a function $f$ is an equality case for $(k,\ZZ_p\times \ZZ_q)$ if and only if $f$ may be written as a tensor product $g\otimes h$, where $g$   is an equality case for $\ZZ_p$, $h$ is an equality case for $\ZZ_q$ and, moreover, one of the two functions $g$ or $h$ is a character  or a Dirac mass.
 \end{theorem}

 We postpone the description of equality cases in the two other cases to the corresponding sections. For these two families one has a complete answer, valid for all values of $k\leq p^2$.

Our results can be summarized as the fact that (except, possibly, for exceptional values for which we were not able to conclude) there are no other equality cases than trivial ones. Unfortunately, even if solutions are simple, proofs are technical and it seems difficult to generalize them to all finite Abelian groups, especially when an arbitrary number of primes $p_j$ is involved.

This paper is a first attempt to show that, even if  Meshulam's Function is  smaller than $k\mapsto|G|-k+1$, there are only a ``small" number of functions such that $|\supp (f)|+|\supp(\widehat f)|\leq |G|$, a phenomenon that is observed in \cite{CRT} in a random setting.

\section{Some preliminary results}

Let us first recall that the function $\theta$ is non increasing. If $\theta (k, G)<\theta (k-1, G)$, then there are equality cases for $(k,G)$. We will see (Remark  \ref{noconverse}, Lemmas \ref{k>q} and \ref{k>p}) that the converse is not true.

\medskip

Next we have the following lemma.
\begin{lemma}\label{dim1} The set $E(k, \theta (k, G))$ is contained in  a finite union of vector spaces of dimension $1$.
\end{lemma}
\begin{proof}
   The set $E(k, \theta (k, G))$ is  contained in the union of $E(A, B)$, where $A$ and $B$ are respectively subsets of $G$ and $\widehat G$,  verify $|A|\leq k$, $|B|=\theta (k,G)$, and
 \begin{equation}\label{Eab}
  E(A,B):=\Big\{f \in L(G); \  \supp(f) \subset A , \; \supp(\widehat{f}) \subset B  \Big\}.
\end{equation}
 Assume $E(A,B)$ is of dimension $\geq 2$. Then we can find $f$ and $g$ two linearly independent functions in $E(A,B)$ and there exists
 a nonzero linear combination of  $f$ and $g$ whose Fourier transform vanishes at some $b\in B$. This implies that  $\theta(k,G)\leq l-1$.
\end{proof}
Note that all equality cases are (non explicitly) known as soon as we  know all subsets $A$ and $B$ for which the space of functions with support in $A$ and spectrum in $B$  is not reduced to $\{0\}$.
\begin{lemma}\label{dim1b} For $|A|=k$ and $|B|=|G|-k+1$, the space  $E(A, B)$ is not reduced to $\{0\}$. As a consequence, $\theta(k, G)\leq |G|-k+1$.
\end{lemma}
\begin{proof} The function $f=\sum_{x\in A} c(x)\mathds{1}_{\{x\}}$ belongs to $E(A,B)$ if the $k$ coefficients $c(x)$ satisfy the $k-1$ linear equations given by $\widehat f (y)=0$ for $y\notin B$. There is at least one nonzero solution to this system.
\end{proof}
There is no equality in general between $\theta(k, G)$ and $|G|-k+1$ for $k\neq 1, |G|$, except for $|G|$  a prime number or when  $ k=|G|-1$. This is described in the next proposition, as well as the values of $k$ for which $\theta(k, G)=2$.
\begin{proposition}\label{caseminus1} Let $G$ be a finite Abelian group $G$ such that $|G|$ is not prime. Let $1<d_1<\cdots<d_r<|G|$ be the divisors of $|G|$. Then \begin{itemize}
\item [(i)]
$\theta (k, G)=2$ if and only if $|G|-d_r\leq k \leq |G|-1$.
\item [(ii)]Moreover, there are equality cases for $(|G|-1, G)$ if and only if the group $G$ is cyclic.
\item [(iii)] One has the inequality $\theta (k, G)\leq |G |-k$ for $1<k<|G|-1$.
\end{itemize}
\end{proposition}
\begin{proof} Let us first prove that $\theta (|G|-1, G)\geq 2$. Indeed, when the Fourier transform of $f$ is a Dirac mass, then $f$ is a character and its  support  has cardinality $|G|$.

Let us next consider values of $k$ for which  $\theta (k, G)=2$. We first compute the cardinality of the support of all functions $f$ that vanish at one point at least and  such that $\widehat f$ is a linear combination of two Dirac masses with nonzero coefficients.
Up to translation, modulation and multiplication by a constant, we can assume that $ f:=1-\chi$, with $\chi\in \widehat G$ a non-principal character.
Then $f$ is supported by the set of all $x\in G
$ such that $\chi(x)\neq 1$.
The complement of $\supp(f)$ is a subgroup $H$ of $G$, so that the only  possible values for $|\supp(f)|$ are $|G|-d_r, \ldots ,|G|-d_1,|G|-1$. If $k<|G|-d_r$, there does not exist any non-principal character $\chi$ that takes the value $1$ on a set of cardinality larger than $|G|-k$. So $\theta (G, k)>2$.

 Let us prove that $\theta (|G|-d_r, G)=2$. There exists a subgroup $H$ of cardinality $d_r$. The group $G/H$ has cardinality $|G|/d_r$, which is a prime number. So $G/H$ is a cyclic group and any non-principal character $\chi_0$ on $G/H$ takes the value $1$ at the neutral element $0$ only. It extends into a character $\chi$ on $G$ by the formula $\chi(x):=\chi_0(\dot{x})$, where $\dot{x}$ is the equivalence class of $x$ mod $H$. The character $\chi$ has the property that it takes the value $1$ exactly on $H$. The corresponding function $ f:=1-\chi$ belongs to $E_0(|G|-d_r, 2)$. So $\theta (|G|-d_r, G)\leq 2$. By monotonicity, for $|G|-d_r\leq k \leq |G|-1$ we have $2\leq\theta (|G|-1, G)\leq \theta (k, G)\leq \theta (|G|-d_r, G)\leq 2$.   Therefore,(i) follows at once.

To see (ii), it is sufficient to prove that there exists a character $\chi$ such that $\chi(x)\neq 1$ for $x\neq 0$ if and only if $G$ is cyclic. Let us first assume that $G$ is cyclic, that is, $G=\ZZ_n$ for some positive integer $n$. Then the character that identifies with $1$ by \eqref {character} has the required property. Conversely, if $\chi$ is such a character, it takes $|G|$ different values since $\chi(a)\chi(b)^{-1}=\chi(a-b)\neq 1$ for $a\neq b$.  It follows that that the
range of $\chi$ consists of all order $|G|$ roots of unity,
whence $\chi^m$ is different from the principal character $1$ for $m<|G |$, while $\chi^{|G|}=1$. This implies that the character $\chi$ generates the whole group $\widehat G$. So $\widehat G$, equivalently $G$, is cyclic.

    Let us now prove (iii). We have just proved that the inequality is satisfied for $|G|-d_r\leq k<|G|-1$.  For $d_r\leq k\leq |G|-d_r$, we use the inequalities $\theta(k, G)\leq \theta(d_r, G)=d_1\leq |G|-d_r\leq |G|-k$. The first inequality is a consequence of monotonicity. Then the equality comes from equality cases in Donoho-Stark's Uncertainty Principle. We next use the fact that $d_1+d_r\leq 2d_r\leq d_1d_r= |G|$.
 It remains to consider $2\leq k<d_r$. This follows from the fact that $\theta (k, G)\leq\theta (2, G)=|G|-d_r$.
\end{proof}
\begin{remark}\label{noconverse} So there are equality cases for $(|G |-1, G)$ when G is the group $\ZZ_p\times \ZZ_q$,
with $p, q$ two distinct prime numbers,  or when $G$ is the group $\ZZ_{p^2}$. This proves that for $k=|G|-1$ one may have simultaneously the equality $\theta(k, G)= \theta(k-1, G)$ and the existence of  equality cases for $(k, G)$. Note that in the group $\ZZ_p\times \ZZ_p$ every element generates a proper subgroup and there is no equality case for $k=p^2-1$.\end{remark}
\begin{remark}\label{onlycase} For the same reasons as above, there are  equality cases for $(|G |-d_j, G)$ if and only if there exists a character $\chi$ such that the subgroup $\{x\in G\ ; \chi(x)=1\}$ has cardinality $d_j$. In particular, when G is the group $\ZZ_p\times \ZZ_q$ with $p>q$, there are equality cases for $((p-1)q, G)$. These equality cases can be written as $f\otimes \xi$, with $f$ an equality case for $(p-1, \ZZ_p)$ and
 $\xi$ a character on $\ZZ_q$.
\end{remark}

The next lemma allows to exchange the role of $f$ and $\widehat f$.
\begin{lemma}\label{reciproque}
Assume that $ \theta(k,G)< \theta(k-1,G)$. Then
$$\theta (\theta (k, G),\widehat G)=k.$$ Moreover, $f$ is an equality case for $(k, G)$ if and only if its Fourier transform is an equality case for $(\theta (k, G), \widehat G)$.

\end{lemma}
The proof is elementary and we leave it to the reader.

\medskip

We next give all equality cases for a product with a supplementary assumption.

\begin{proposition}\label{product} Let $G=G_1\times G_2$ and $1\leq k\leq |G|$. Then
\begin{equation}\label{expression-theta}
\theta(k, G)=\min \{\theta (k_1, G_1)\theta(k_2, G_2) ;\; k_1 k_2\leq k, \ \ 1\leq k_i\leq |G_i|,\; i=1,2  \}.
\end{equation}
Assume that $(k_1, k_2)$ is the only pair for which $k_1k_2\leq k$ and
\begin{equation}\label{equal}
\theta (k, G)=\theta (k_1, G_1)\theta(k_2, G_2).
\end{equation} Then there are equality cases for $(k,G)$ if and only if  $k=k_1k_2$
and there are equality cases for $(k_i,G_i)$, $i=1, 2$. Moreover, all equality cases for $(k,G)$  may be written as $f_1(x_1)f_2(x_2)$, with $f_i$ an equality case for $(k_i,G_i)$, $i=1, 2$.
\end{proposition}
\begin{proof}
It is inspired by Meshulam's paper, who has proved  the first statement.  Let $f$ be a nonzero function with support of size $\leq k$ and spectrum of size $\theta(k, G)$. For $\chi(x)=\chi_1(x_1)\chi_2(x_2)$ a character, that is, an element of $\widehat G=\widehat {G_1}\times \widehat {G_2}$, we write
\begin{eqnarray*}
\widehat f(\chi_1, \chi_2)&=&\sum_{x_1\in G_1} \sum_{x_2\in G_2}f(x_1,x_2)\chi_1(-x_1)\chi_2(-x_2)\\
&=&\widehat F_{\chi_1} (\chi_2).
\end{eqnarray*}
 Here
$$F_{\chi_1} (y)=\sum_{x_1\in G_1} f(x_1,y)\chi_1(-x_1)= \widehat{f_y}(\chi_1),$$
if we put $f_y(x_1):=f(x_1,y)$ for $y\in G_2$. Then
$$|\supp (\widehat f(\chi_1, \cdot))|\geq \theta (|\supp (F_{\chi_1})|, G_2)$$ when $F_{\chi_1}\neq \textbf{0}$.
Let us denote  $$\widehat S :=\{\xi\in \widehat {G_1}; \; F_\xi\neq \textbf{0}\}= \{\xi\in \widehat {G_1}; \; \widehat f(\xi, \cdot)\neq \textbf{0}\}.$$ Then we have
\begin{equation}\label{equ1}
    |\supp (\widehat f(\cdot, \cdot))|\geq |\widehat S|\min_{\xi \in \widehat S}\theta (|\supp (F_\xi)|, G_2).
\end{equation}
Now take for $t$ the size of $T:=\{y;\; f_y\neq \textbf{0}\}$. The support of $F_\xi$ is contained in $T$ for all $\xi\in \widehat{G_1}$, so that,
 for $\xi\in \widehat S$, we have
\begin{equation}\label{equ2}
 \theta (|\supp (F_\xi)|, G_2)\geq \theta (t, G_2).
\end{equation}
We also have
\begin{equation}\label{equ3}
|\widehat S|\geq |\cup_{y\in T}\{\xi;\; F_\xi(y)=\widehat {f_{y}}(\xi)\neq 0\}|\geq \theta(s,G_1)
\end{equation}
for $s$ the smallest size for the support of $f_y$.  We finally remark that $st\leq k$. So we conclude from \eqref{equ1}, \eqref{equ2}, \eqref{equ3} that
\begin{equation}\label{equ4}
\theta (k,G)\geq \theta (s,G_1)\theta(t, G_2).
\end{equation}
We have proved that
$$\theta (k, G)\geq\min \{\theta (k_1, G_1)\theta(k_2, G_2) ;\; k_1 k_2\leq k, \ \ 1\leq k_i\leq |G_i|,\; i=1,2  \}.$$
Next we prove that there is equality in this inequality. Assume that the minimum is obtained for $k_1, k_2$.
 Let $f_1\in L(G_1)$ and $f_2\in L(G_2)$ such that $ |\supp (f_i)|\leq k_i$ and
 $ |\supp (\widehat f_i)|=\theta (k_i,G_i)$ for $i=1, 2$. Then $f_1\otimes f_2$ has support of size $\leq k_1k_2\leq k$ and its spectrum has size
$\theta (k_1,G_1)\theta(k_2, G_2)= \theta(k, G)$.

Next, assume that $(k_1, k_2)$ is the only pair for which $k_1k_2\leq k$ and \eqref{equal} is valid.
 Let us characterize the values $k$ for which we have equality. Assume that there is some equality case $f$ for $(k, G)$.
 If we proceed as above, the inequality \eqref{equ4} is an equality and the minimum is obtained for $(s, t)$, which coincides with $(k_1, k_2)$.
 Thus inequalities \eqref{equ3}, \eqref{equ2} and \eqref{equ1} are also equalities. Looking at the definition of $T$ and $\widehat S$, it is easily seen  that $T$ is the projection of the support of $f$ on $G_2$ while $\widehat S$ is the projection of the support of $\widehat f$ on $\widehat G_1$. So  $t$ is the size of the projection $T$ of $\supp (f)$ on $G_2$, while $\theta (s, G_1)$ is the size of the projection $\widehat S$ of $\supp (\widehat f)$ on $\widehat G_1$. Exchanging the role of $G_1$ and $G_2$, we define as well $S$ and $\widehat T$, which are respectively of size $s$ and $\theta (t, G_2)$. In particular, the size of $\supp (f)$, which is contained in $S\times T$, is at most $st$. This proves that $k=st$ and the support of $f$ is exactly $S\times T$. Similarly the support of $\widehat f$ is exactly $\widehat S \times \widehat T$. Moreover,  each $f_y$ has the same support $S$ and the same spectrum  $\widehat S$. It is in particular  an equality case for $(s, G_1)$. By symmetry, there are also equality cases for $(t, G_2)$, with support $T$ and spectrum $\widehat T$. More precisely, there exists some function $h_1$ on $G_1$ (resp. $h_2$ on $G_2$) with support $S$ and spectrum $\widehat S$ (resp. $T$ and $\widehat T$).  Then $h_1\otimes h_2$ is an equality case for $(st, G)$, with support $S\times T$ and spectrum $\widehat S \times \widehat T$. So, both $f$ and $h_1\otimes h_2$ belong to $E( S \times  T, \widehat S \times \widehat T)$. By Lemma \ref{dim1}, we have $f=ch_1\otimes h_2$. We have proved that $f$ can be written as a tensor product.

This finishes the proof of the proposition.
\end{proof}

\section{The case of groups $\ZZ_{p}\times \ZZ_q$, with $q < p$ prime numbers}
 Let us first give Meshulam's Function, which one can already find in \cite{DVB2}. We give the proof, nevertheless, since we want to know when there is uniqueness of the minimum.
 \begin{proposition}\label{meshtwoprimes}
Let $G=\ZZ_{p}\times \ZZ_q$, with $p$ and $q$ prime numbers such that $1<q<p$. Then

$\theta(k,G)=
\begin{cases}
p(q-k+1),& 1\leq k \leq q;\\
 p - [\frac{k}{q}]+1,& q\leq k \leq q\,\frac {p+1}{q+1};\\
q(p-k+1),&q\,\frac {p+1}{q+1}\leq k \leq p;\\
q-[\frac{k}{p}]+1,& p\leq k \leq pq.
 \end{cases}$
 \end{proposition}
 \begin{proof} Using Tao's Theorem and Proposition \ref{product}, we know that
$$
\theta(k,G)=\min\{(p-s+1)(q-t+1);\;st\leq k,\; 1\leq s\leq p,\; 1\leq t \leq q\} .
$$
Let us first consider the function $F(s,t):=(p-s+1)(q-t+1)$ with real variables $s, t$. Let $R$ be the rectangle defined by $1\leq s\leq p$, $1\leq t\leq q$
and let  $\Delta:=\Delta_k$ be the region in $R$ such that $st\leq k$. We first look at the minimum of $F$ on $\Delta$.   Because of the concavity of $F$ its minimum cannot be attained inside the domain $\Delta$.  Moreover, the function is a concave function of $s$ on the hyperbola $st=k$, so that it does not attain its minimum in the interior part of the hyperbola, but on its boundary. It decreases on the common boundary with  $R$ when $s$ or $t$ increases, so that the minimum is obtained on the intersection of the hyperbola with the boundary of $R$, which consists of two points (we assume that $k$ is different from  $1$ or $pq$, where the conclusion is immediate). We have to consider separately four cases. In the following table, we give for each of them the values of the minimum, followed by the two points of intersection:
$$\min_{\Delta} F(s,t)=\begin{cases}
p(q-k+1)\quad  \mbox {\rm obtained in}\; (1, k)\; \mbox {\rm or} \;(k, 1),\quad & 1\leq k \leq q;\\
 p - \frac{k}{q}+1\quad \mbox {\rm obtained in}\; (\frac kq, q)\; \mbox {\rm or}\; (k, 1),\quad & q\leq k \leq q\,\frac {p+1}{q+1};\\
q(p-k+1)\quad \mbox {\rm obtained in}\; (\frac kq, q)\; \mbox {\rm or} \;(k, 1),\quad &q\,\frac {p+1}{q+1}\leq k \leq p;\\
q-\frac{k}{p}+1\quad \mbox {\rm obtained in}\; (\frac kq, q)\; \mbox {\rm or} \;(p, \frac kp),\quad & p\leq k \leq pq.
 \end{cases}
 $$
 Let us remark that the minimum is obtained at exactly one point, except when $k= q\,\frac {p+1}{q+1}$. We give also the other point  because it will play a role later on.

 Let us now consider the  minimum on the integer values for $s$ and $t$, that is, $\theta(k, G)=\min_{\Delta\cap (\NN\times \NN)} F(s,t)$. When the minimum on  the whole $\Delta$ is obtained for integer values of $s$ and $t$, it is also the minimum on $\Delta\cap (\NN\times \NN$). This allows to conclude for the first and the third case.

 Let us concentrate on the second case, for which the minimum on $\Delta$ is obtained at $(\frac kq, q)$.  It coincides with the minimum on
 $\Delta\cap (\NN\times \NN)$ when $k$ is a multiple of $q$. It remains for this second case to consider values of $k$ such that $k$ is not a multiple
 of $q$, with $q<k< q\,\frac {p+1}{q+1}$. Then the minimum on $\Delta$ is not an integer and cannot be attained on $\Delta\cap (\NN\times \NN$).
 So the minimum on $\Delta\cap (\NN\times \NN)$ is at least  the smallest integer that is larger than $\min_{\Delta} F(s,t)$, that is, $p - [\frac{k}{q}]+1$. It remains to see that this value is really attained, which is the case at the point $([\frac kq], q)$. So $\theta(k, G)=p - [\frac{k}{q}]+1$.

 The same argument allows to conclude for the fourth case as well.
\end{proof}

In view of the use of Proposition \ref{product} we try to answer the following question. Is there uniqueness of the pair of integers $(s, t)$ for which the minimum is obtained? Let us give the following definition.
\begin{definition}
We call $\mathfrak{M}$ the set of integers $k$ such that $1< k< pq-1$ and such that the minimum of $F$ over $\Delta_k\cap (\NN \times \NN)$ is obtained at exactly one point $(s, t)$.
\end{definition}

Clearly $k$ belongs to $\mathfrak{M}$ when the minimum  of $F$ over $\Delta_k\cap (\NN \times \NN)$ coincides with the minimum over all $\Delta$ and when there is uniqueness for this last one. This excludes from $\mathfrak M$ the integer  $k= q\frac {p+1}{q+1}$ when $p+1$ is a multiple of $q+1$. But we directly conclude that all $k\leq q$, as well as all $k$ such that $q\frac {p+1}{q+1}< k \leq p$ belong to $\mathfrak M$. In the two remaining intervals, multiples of $q$ when $q<k< q\frac {p+1}{q+1}$ and multiples of $p$ when $p<k<pq$ belong to $\mathfrak M$. The next two lemmas  give the complete description of $\mathfrak M$.

\begin{lemma}\label{k>q} When $q<k < q\,\frac {p+1}{q+1}$, the minimum of $F$ on $\Delta_k\cap (\NN \times \NN)$ is attained at one point exactly except perhaps for one exceptional value. More precisely, there exists a value of $k$ for which the minimum is attained at two points if and only if $p+1\equiv q$ mod $q+1$.  In this case, the exceptional value is $k=q\left[\frac {p+1}{q+1}\right]+q-1$, for which  the minimum is attained at $([\frac kq], q)$ and $(k,1)$. For this particular value of $k$, there are equality cases of the form $f\otimes \delta_{\ZZ_q}(\cdot-b) $,
 with $b\in\ZZ_q$ and $f\in L(\ZZ_p)$ such that $|\supp(f)|=k$ and $|\supp(\widehat f)|=p-k+1$.
\end{lemma}
 \begin{proof}
The minimum $p-[\frac kq]+1$ is attained at only one point $(s,t)=([\frac kq], q)$ on the line $t=q$. Assume that it is attained at some other point $(s', t')$. Then $t'\leq q-1$. For the same reasons as before this can only occur for $t'=1$ or $t'=q-1$, that is, at one of the points $(k, 1)$ and $([\frac k {q-1}], q-1)$. But $\theta (k, G)$ is attained at the point $(k,1)$ if and only if $q(p-k+1)=p-[\frac kq]+1$. If we write $k=qj+r$, with $0\leq r<q$, this means that $qr=(q-1)(p-(q+1)j+1)$, which implies that $r=q-1$. Moreover, $p+1=(q+1)j+q$, so that $p+1\equiv q$ mod $q+1$, and $j=\left[\frac {p+1}{q+1}\right]$. So $k=q\left[\frac {p+1}{q+1}\right]+q-1=q\frac {p+1}{q+1}+q-1$.

 For $q=2$ this concludes the proof. For $q>2$ it suffices to show that  $F(\dfrac{k}{q-1},q-1)>
F\left(\left[\dfrac{k}{q}\right]\right)$, that is, $p+1> 2\left[\dfrac{k}{q-1}\right]-\left[\dfrac{k}{q}\right]$.
This follows from the inequalities $2\left[\dfrac{k}{q-1}\right]-\left[\dfrac{k}{q}\right]-1\leq \dfrac{k}{q-1}<\dfrac{q(p+1)}{q^2-1}<p$. We have used the assumption on $k$ and the fact that $q>2$.

 Finally, it is easy to see that functions $f\otimes \delta_{\ZZ_q}(\cdot-b)$ are equality cases.
\end{proof}

\begin{lemma}\label{k>p}
Let $k=jp+r$, with $1\leq j\leq q-1$ and $1\leq r\leq p-1$. Then $k$ belongs to $\mathfrak M$ if and only if $r<(p-q)(q-j)$. When $k\notin \mathfrak{M}$, then the minimum is  attained at the points $(p, j)$ and $([\frac kq], q)$. If, moreover, $k=lq$, then the minimum is attained at $(l, q)$ if and only if $(p-l)(p-q)<p$ and, in this particular case, there are equality cases of the form $ f\otimes\chi $, with $\chi$ a character of $\ZZ_q$ and $f\in L(\ZZ_p)$ such that $|\supp(f)|=l$ and $|\supp(\widehat f)|=p-l+1$.
\end{lemma}
\begin{proof}
 We  have  $\theta (k, G)=q-j+1 =ab$, with $p+1-s=a$ and $q+1-t=b$. Moreover, we have $st\leq k<(j+1)p$, that is,
$$(p+1-a)(q+1-b)<p(q+2-ab).$$
This implies that $a<1+\frac p{(p+1)(b-q-1)}$ so that either $a=1$, or $(p+1)b\leq q+1+p$, which can only happen when $b=1$. The case $a=1$ corresponds to the pair $(p,[\frac kp])$) where we already know that $\theta(k, G)$ is attained, while the case $b=1$  corresponds to the pair $([\frac kq], q)$. So either the minimum is attained at only one point or it is  attained at exactly two points. We finally find that it is also attained at the second point $([\frac kq], q)$ if and only if $p-\frac kq \leq q-j$, that is, $(p-q)(q-j)\leq r$. The particular case $k=lq$ is obtained from elementary computations.
\end{proof}

We do not know whether the equality cases that we have described in the two lemmas \ref{k>q} and \ref{k>p}  are the only ones. We do not know either how to describe all equality cases when the minimum is attained at two points.

Lemmas \ref{k>q} and \ref{k>p} and the comments above may be summarized in the following statement.
\begin{proposition}
\begin{itemize}
 \item [(i)]
 All $k\leq q$ belong to $\mathfrak{M}$.
 \item [(ii)]
All integers $k$ such that $q< k\leq q\frac{p+1}{q+1}$ belong to $\mathfrak{M}$ except possibly one exceptional value of $k$. If $p+1$ is not congruent to $0$ or $-1$ mod $q+1$ there is no exceptional value of $k$ in this interval.
\item [(iii)]
 All $k$ such that $q\frac{p+1}{q+1}<k\leq p$ belong to $\mathfrak{M}$.
 \item [(iv)]
 An integer $ k=jp+r\leq pq-1$  such that $1\leq j\leq (q-1)$ and $0\leq r\leq q-1$ belongs to $\mathfrak M$ if and only if $r<(p-q)(q-j)$.
\end{itemize}
\end{proposition}

We will be able to give a complete description of equality cases for $k\in \mathfrak M$.
We have seen earlier that one can conclude for equality cases for  $k\geq p(q-1)$, for which $\theta (k,G)=2$.  Namely, there are equality cases only when $k=pq-1$ and $k= q(p-1)$ (see Proposition \ref{caseminus1} and Remark \ref{noconverse}).

\begin{theorem}\label{th-pq}
Let $G=\ZZ_{p}\times \ZZ_q$, with $p$ and $q$  prime numbers such that $1<q<p$ and $k\in \mathfrak M$. Then we have the following.
\begin{itemize}
\item [(i)]
  For $k\leq q$,  equality cases are of the form $\delta_{\ZZ_p}(\cdot -a)\otimes f$, with $a\in\ZZ_p$ and $f\in L(\ZZ_q)$ such that
  $|\supp(f)|=k$ and $|\supp(\widehat f)|=q-k+1$.
\item [(ii)]
 Let  $k$ be such that $q\leq k<q\frac{p+1}{q+1}$. There exist equality cases for $(k, G)$ if and only if $k$ is divisible by $q$. When $k=qj$ with $1\leq j<\frac{p+1}{q+1}$, equality cases are of the
 form $ f\otimes\chi $, with $\chi$ a character of $\ZZ_q$ and $f\in L(\ZZ_p)$ such that $|\supp(f)|=j$ and $|\supp(\widehat f)|=p-j+1$.
 \item [(iii)]
For $k$ such that $q\frac{p+1}{q+1}< k\leq p$,  equality cases are of the form $f\otimes \delta_{\ZZ_q}(\cdot-b) $, with $b\in\ZZ_q$ and $f\in L(\ZZ_p)$ such that
$|\supp(f)|=k$ and $|\supp(\widehat f)|=p-k+1$.
\item [(iv)]
   Let  $k$ be such that $p\leq k<pq-1$. There exist equality cases for $(k, G)$  if and only if $k$ is divisible by $p$. When $k=pj$ with  $1\leq j\leq q-1$, then equality cases are of the form $ \chi\otimes f$,
 with $\chi$ a character of $\ZZ_p$ and $f\in L(\ZZ_q)$ such that $|\supp(f)|=j$ and $|\supp(\widehat f)|=q-j+1$.
\end{itemize}
For all other values $k\in \mathfrak M$, there are no equality cases.
\end{theorem}
\begin{proof}
 Proposition \ref{product} allows to conclude directly.
 \end{proof}

 Let us add some remarks.

 \begin{remark} For values of $k$ that do not belong to $\mathfrak M$ there may be equality cases as described in Lemmas \ref{k>q} and \ref{k>p}. We do not know whether they are the only ones, except when $k>p(q-1)$. Let us recall that for $k>p(q-1)$,  $q(p-1)$ is the only value for which there are equality cases. They are described in Remark \ref{onlycase}.
  \end{remark}
  \begin{remark}
  We have seen that  equality cases for $k=pq-1$ may be written as $c\chi(1-\xi(\cdot -a))$, where $\chi$ and $\xi$ are two characters with $\xi$ that generates $\widehat G$, while $a$ is an element of $G$. Let us note that they are not tensor products. This is the only case for which $\theta (k, G)=|G|-k+1$, which may explain the difference of structure of equality cases.
  \end{remark}
  \begin{remark}
  It may be helpful to give another description of the set of values $k$ for which there are equality cases.
  We define \begin{equation}
\mathfrak{M_0}:=\{k\in (1, pq-1)\ ;\ \theta (k, G)<\theta (k-1, G)\}.
\end{equation}
From the definition of $\theta$ given by \eqref{theta} we  a priori know that there are equality cases for $k\in \mathfrak{M_0}$. Moreover, when using Proposition \ref{product}, we see that the minimum in \eqref{expression-theta}is obtained only for $k=k_1 k_2$, with $\theta (k, G)=\theta (k_1, \ZZ_p) \theta (k_2, \ZZ_q)$. From the expressions of $\theta$ given in Proposition \ref{meshtwoprimes}, it follows that all integers $k\in \mathfrak{M_0}$ belong to $\mathfrak{M}$ except for the exceptional value $q\frac {p+1}{q+1}$ (when it is an integer). On the opposite, it may be deduced from Theorem \ref{th-pq} that  there are no equality cases when $k\in \mathfrak M$ does not belong to $\mathfrak{M_0}$.
\end{remark}

 \section{The case of groups $\ZZ_{p}\times \ZZ_p$, with $ p$ prime}\label{p^2}
The formula for Meshulam's Function is also given in \cite{DVB2}.
\begin{proposition}\label{product-gp}
Let $G=\ZZ_p\times \ZZ_p$, with $p$ a prime number. Then

$\theta(k,G)=
\begin{cases}
p(p-k+1),&1\leq k \leq p;\\
p-[\frac{k}{p}]+1,& p\leq k \leq p^2.
 \end{cases}$
\end{proposition}
The proof is the same as for Proposition \ref{meshtwoprimes}. Let us remark that now the minimum is not achieved at one point. For $k<p$ it is attained for both pairs $(1, k)$ and $(k,1)$ (and only there), while, for $k>p$  it is attained for $(p, [\frac{k}{p}])$ and $([\frac{k}{p}],p)$. The same proof as for Lemma \ref{k>p} allows us to prove that there is no other pair for which the minimum is attained.

\medskip

To describe equality cases, we will use the fact that $\ZZ_p\times \ZZ_p$ is a vector space of dimension $2$ over the field $\ZZ_p$. The main difference with the previous case is the fact that there are many proper subgroups of size $p$, namely all proper subgroups generated by one element $m=(m_1, m_2)$, which we write $G_m$. Let us define a scalar product on $\ZZ_p\times \ZZ_p$  (with values in $\ZZ_p)$ by $\langle x, \xi\rangle:= x_1\xi_1+ x_2\xi_2$. Note that  there exists isotropic directions  if and only if $-1$ is a square in $\ZZ_p$  (so, in particular, when $p= 5, 13, 17, \cdots $). In this case, assuming that $p>2$,  the two isotropic directions are given by the vectors $(1, \pm m_0)$, where $m_0$ is such that $-1\equiv m_0^2$ mod $p$.

The orthogonal of $G_m$, which we denote $G_m^{\perp}$, is $G_{\widetilde m}$, with $\widetilde m=(m_2,-m_1)$. If $m$ is non isotropic, then $\ZZ_p\times \ZZ_p$ can also be written as $G_m\times G_{\widetilde m}$. If $m$ is isotropic, then  $m$ and $\widetilde m$ are collinear and
$G_m=G_{\widetilde{m}}$.

Let us   consider the  linear transformation $A(m)$, given for $m\in\ZZ_p\times \ZZ_p$ non isotropic  by the invertible matrix (that we still denote $A(m)$)
\[ A(m):=\left( \begin{array}{cc}
m_1 & -m_2 \\
m_2 & m_1\end{array} \right).\]
Then $A(m)^*=\left( \begin{array}{cc}
m_1 & m_2 \\
-m_2 & m_1\end{array} \right)=A(\overline m)$, with $\overline m = (m_1, -m_2)$,
and \\ $A(m)A(m)^*=(m_1^2+m_2^2)I$. Moreover
$$\langle A(m)x, A(m)\xi\rangle=(m_1^2+m_2^2)\langle x, \xi\rangle.$$ Let us identify $\widehat G$ with $G=\ZZ_p \times \ZZ_p$ by \eqref{character}. With this identification,  matrices $A(m)$ act on characters.   In particular, when $m$ is
non-isotropic, i.e. when $A(m)$ is nonsingular, let us write $g(x):=f(A(m)^{-1} x)$,
so that $g$ is the image of $f$ under the action of $A(m)$. Its Fourier transform is $\widehat g(\xi)= \widehat f(A(m)^{*}\xi)$.
So, when $m$ is non isotropic, the transformation $A(m)$ preserves the sizes of the support and the spectrum of a function, so that the sets $E(k, l)$ and $E_0(k, l)$ are invariant under
the action of $A(m)$.

Here is a simple way to describe all  equality cases, which do not distinguish between  values of $m$. Let us first consider $k<p$.
For each subgroup $G_m$, the following lemma describes  equality cases supported by $G_m$.

\begin{lemma}\label{G_m} Let $m\in \ZZ_p\times \ZZ_p$ and let $f$ be a function supported in $G_m$ and $k<p$. Then $f$ is an equality case for $(k, \ZZ_p\times \ZZ_p)$ if and only if the function $h$ defined on $\ZZ_p$ by  $h(j):=f(jm)$ for $j\in \ZZ_p$ is an equality case for $(k, \ZZ_p)$.
\end{lemma}
\begin{proof}
Let $f$ be an equality case supported on $G_m$ and $h$ as above. Elements in $G_m$ may be written as $jm=(jm_1, jm_2)$, with $j\in \ZZ_p$, where each product has to be interpreted as a product in $\ZZ_p$.

  The Fourier transform of $f$ is given by
$$\widehat f(\xi)=\sum_{j\in \ZZ_p} e^{-2i\pi \frac{j\langle \xi, m\rangle}{p}} h(j).$$
As above in \eqref{character}, we have identified the character $\xi$ with an element of $\ZZ_p\times \ZZ_p$ in this expression.
In particular $\widehat f$ is constant on cosets of $G_m^\perp$. Since $|\supp (\widehat f)|=p(p-k+1)$ by Proposition \ref{product-gp}, the support of $\widehat f$ is the union of $p-k+1$ cosets. Assume that $G_m$ is not generated by $(0, 1)$ (which we can always ensure by exchanging the coordinates if necessary), so that $G_m^\perp$ and $\ZZ_p\times\{0\}$ are not equal. Then $\widehat f(l,0)= \widehat h(m_1l )$ is non zero for exactly $p-k+1$ values of $l$. Since $p$ is prime, multiplication by $m_1$ is a bijection on $\ZZ_p$ and $|\supp (\widehat{h})|=p-k+1$.
\end{proof}
If $m$ is non isotropic then  $f$ is the image under $A(m)$ of  $h\otimes \delta_{\ZZ_p}$.

 We can now state  the theorem, which says that all equality cases can be obtained from these examples. As in the previous case, we have already considered the value $k=|G|-1$ in Remark \ref{noconverse}.
\begin{theorem}
Let $G=\ZZ_p\times \ZZ_p$, with $p$ a prime number.
There are equality cases if and only if $\theta (k, G)<\theta (k-1, G)$. They can be described as follows.
 \begin{itemize}
 \item [(i)]
 For all $k\leq p$,  equality cases  are  translates of equality cases supported by some subgroup $G_m$. In particular, when $m_1^2+m_2^2$ is not equivalent to $0$ mod $p$, they are transforms under the transformation $A(m)$ of a function of the form $f\otimes \delta_{\ZZ_p}(\cdot -a)$,
 with $a\in\ZZ_p$ and $f\in L(\ZZ_p)$ such that $|\supp(f)|=k$ and $|\supp(\widehat f)|=p-k+1$.
\item  [(ii)]
For all $p\leq k \leq  p^2-1$, there are equality cases if and only if $k$ is divisible by $p$. For  $k=pj$, their Fourier transforms are equality cases for $p-j+1$.
\end{itemize}
\end{theorem}
\begin{proof}
 We have seen that the functions given in the statement are equality cases. Let us prove that they are the only ones.  We assume first that $k<p$ and consider an equality case $f$. Without loss of generality we can assume that $f(0)\neq 0$. By Lemma \ref{G_m} it is sufficient to prove that $f$ is supported by some subgroup $G_m$. We define $ T$ as before as $\Pi_2(\supp f)$, where $\Pi_2$ is the projection $(l_1, l_2) \mapsto l_2$. By the same proof as in Proposition \ref{product}, we see that  $|T|$  take the values $1$ or $k$.  Note that the transform of $f$ under the action of $A(m)$ is also an equality case when $m$ is non isotropic, hence the same conclusion holds for all transforms $A(m) \supp f$ with $m$ non isotropic. We conclude from the next lemma, which we use for $E:=\supp f$.

 \begin{lemma} \label{geo} Let $2\leq \nu<p$ and $0\in E\subset \ZZ_p\times \ZZ_p$ with $|E|=\nu$  be given. Assume that $|\Pi_2(A(m)E)|=1$ or $\nu$ for all non isotropic $m\in \ZZ_p\times \ZZ_p$. Then, for some $m\in \ZZ_p\times \ZZ_p$, the set $E$ is contained in  $G_m$.
 \end{lemma}
 \begin{proof}
Let us first assume that there is a non isotropic element $m\in E$ and prove that $E$ is contained in $G_m=A(m)G_{(1, 0)}$. Equivalently, we want to prove that the set $A(\overline m)E$ is contained in $\ZZ_p \times \{0\}$. We already know that the point $(m_1^2+m_2^2, 0)$ belongs to $A(\overline m)E$. This implies that  $|\Pi_2(A(\overline m)E)|<\nu$, so that, by assumption, $|\Pi_2(A(\overline m)E)|=1 $. The conclusion follows at once. We are done when there are no isotropic directions.

 Let us now assume that there exists  two isotropic directions, $(1, \pm m_0)$. We want to prove that, if there is no isotropic elements in $E$, and hence $E$ is contained in the union of $G_{(1, m_0)}$ and $G_{(1, -m_0)}$, then $E$ is contained either in $G_{(1, m_0)}$ or in $G_{(1, -m_0)}$. So finally assume for a contradiction that there exists some $a:=(l_1,l_1m_0)\in E$ and also some $b:=(l_2,-l_2m_0)\in E$ with $l_i \not \equiv 0$. Let us take now $m:=(l_2-l_1, m_0(l_1+l_2))$, which is non isotropic for $m_1^2+m_2^2 \equiv (l_2-l_1)^2 - (l_1+l_2)^2 \equiv -4 l_1 l_2 \not \equiv 0$ mod $p$. An easy calculation yields $\Pi_2( A(m) a) = \Pi_2 (A(m) b) = 2 m_0 l_1 l_2$, hence $\Pi_2$ cannot be bijective on $A(m)E$ and so $|\Pi_2 A(m)E|=1$. This implies that $ A(m)E\subset \ZZ_p\times \{0\}$ or, which is equivalent, $E\subset  A(m)^{-1}G_{(1, 0)}=G_{\overline m}$. This is a contradiction to the working assumption that $E$ does not contain a non isotropic element,  hence this case cannot occur and the lemma is proven.

 \end{proof}

Let us now consider $p\leq k\leq p^2$.  For $k=pj$,  equality cases are deduced from the ones of $p-j+1$ by taking Fourier transforms,
  and we recognize the functions given in the statement of the theorem. Let us prove that there are no equality cases when $k$ cannot be divided by $p$.
  Assume that $pr\leq k<p(r+1)$ and $f$ is an equality case for $k$. Then, proceeding as in Proposition \ref{product} and defining  $\widehat T=\Pi_2 (\supp \widehat f)$
  as before, we find again that  $|\widehat T|$  takes the values $1$ or $p-j+1$. This is also valid for the supports of transforms of
   $\widehat f$ through all transformations $A(m)$ with $m_1^2+m_2^2\neq 0$. So  the support of $\widehat f$ satisfies the assumptions of Lemma \ref{geo}, with $p-j+1$ in place of $\nu$.  This implies that $\widehat f$ is supported in some $G_m$. The support of its Fourier transform, that is, $\supp f$ as well, has a cardinality that is a multiple of $p$.
    This proves that $k$ is a multiple of $p$.
\end{proof}

\section{The case of groups $\ZZ_{p^2}$, with $p$ prime}
As remarked in \cite{DVB2}, the functions $\theta (k, \ZZ_{p^2})$ and $\theta (k, \ZZ_p\times \ZZ_p)$ are identical (and equal to the function $u(k, G)$). We will see that the values of $f$ for which there are equality cases are the same except for $p^2-1$, for which, by Remark \ref{noconverse}, there are no equality cases for $\ZZ_p\times \ZZ_p$ while there are equality cases for $\ZZ_{p^2}$.
\smallskip

Let us  denote by $H$ the unique proper subgroup  of $G$, which is generated by the equivalence class of $p$ in $\ZZ_{p^2}$. The subgroup $H$  identifies with $\ZZ_p$ under the mapping $\ZZ_p \mapsto H$ that maps the congruent class $j=\dot l$ mod $p$ to the congruent class of $pl $ mod $p^2$. For simplification, we identify an element of $\ZZ_p$ with its representative in the interval $[0, p)$ and an element of $\ZZ_{p^2}$ with its representative in $[0, p^2)$ and we denote by $pj$ this element of $\ZZ_{p^2}$.

We will also use equivalent classes modulo $H$. Let us note that, if we identify $\ZZ_{p^2}$ with its representative that lies between $0$ and $p^2-1$, these $p$ equivalent classes may be described as
$$a+H =\{ a+pj\; ;\; j=0, 1,\cdots p-1\}$$
for $a=0, 1, \cdots p-1$.

 The following lemma describes equality cases for functions that are supported in $H$.
\begin{lemma}\label{H} Let $f$ be a function supported in $H$ and $k<p$. Then $f$ is an equality case for $(k, \ZZ_{p^2})$ if and only if the function $h$ defined on $\ZZ_p$ by  $h( j):=f(pj)$ for $j\in \ZZ_p$ is an equality case for $(k, \ZZ_p)$.
\end{lemma}
\begin{proof}
Let $f$ be an equality case for $(k, \ZZ_{p^2})$ supported in $H$, and $h$ as above, so that $|\supp h|=k$. A character $\xi$ on $\ZZ_{p^2}$ identifies with an element of $\ZZ_{p^2}$, which identifies with some $\alpha + p \eta$, with $\alpha, \eta =0, 1, \cdots, p-1$. If we denote by $\widehat h$ the Fourier transform of the function $h$ on $\ZZ_p$, then
$$ \widehat f(\alpha+p\eta)=\sum_j e^{-2i\pi \frac{\alpha j}p}h(j)=\widehat h(\alpha).$$ The cardinalities of the supports of $\widehat f$ and $\widehat h$ are such that  $p(p-k+1)=|\supp (\widehat f)|=p|\supp(\widehat h)|$. So $|\supp(\widehat h)|=p-k+1$ and $h$ is an equality case  for $(k, \ZZ_{p})$. Conversely such a function is an equality case.
\end{proof}
Translates of such equality cases, as well as Fourier transforms, are equality cases. We shall prove that they are the only ones, except when $k=p^2-1$. Recall that for $k=p^2-1$, by Proposition \ref{caseminus1} (ii) the Fourier transforms of the equality cases are of
 the form $\alpha\chi ( \mathds{1}_{\{x\}}- \mathds{1}_{\{y\}})$, with $\alpha$ a constant, $\chi$ a character and $x,y$ two points such that $x-y\notin H$.

 We will assume, from now on, that $k<p^2-1$. We have the following theorem.
\begin{theorem} \label{psquare}
Let $G=\ZZ_{p^2}$, with $p$ a prime number. Then

$\theta(k,G)=
\begin{cases}
p(p-k+1),&1\leq k \leq p;\\
p-[\frac{k}{p}]+1,& p\leq k \leq p^2.
 \end{cases}$

Moreover, when $k<p^2-1$, there are equality cases if and only if $\theta (k, G)<\theta (k-1, G)$. They can be described as follows.
 \begin{itemize}
  \item [(i)]
 For all $k\leq p$,  equality cases  are of the form
 \begin{equation}\label{case1}
 f(px+x')=g(x)\ \mathds{1}_{\{a\}}(x'), \ \ \ x,\;x'=0, \ldots, p-1,
\end{equation}
  with $g\in L(\ZZ_p)$ such that $|\supp(g)|=k$ and $|\supp(\widehat g)|=p-k+1$, and $a$ one of the values $0, \ldots, p-1$.
\item [(ii)]
For all $p\leq k < p^2-1$, there are equality cases if and only $k$ is divisible by $p$. For  $k=pj$, equality cases are of the form
  \begin{equation}\label{case2}
 f(px+x')=\chi (x)g(x'), \ \ \ x,\;x'=0, \ldots, p-1,
\end{equation}
with $\chi$ a character of $\ZZ_p$ and
 $g\in L(\ZZ_p)$ such that $|\supp(g)|=j$ and $|\supp(\widehat g)|=p-j+1$.
  \end{itemize}
\end{theorem}

 \begin{proof}
 Even if not stated in the same way, most of this theorem is  proved in \cite{DVB1} (and even its analog for any arbitrary power of $p$) but using a different vocabulary, with decomposition of Fourier matrices that are not simple to follow from a group point of view. So we give a complete proof in our vocabulary.

  As in the case of  products of groups, the computation of $\theta$ is given by Meshulam in \cite{M}. We nevertheless give  some details of the proof, which we will use again for
   equality cases. For $f$  a nonzero function such that $|\supp(f)|\leq k$, we define $s$ and $t$ as follows: $t$ is  the number of $a=0, \ldots, p-1$ such that the function $f_a$ defined by $f_a(x):=f(a+px)$, is not identically $0$ on $\ZZ_p$ ; $s$ is the minimum of $|\supp (f_a)|$. Clearly $st\leq k$. It is then proved in \cite{M}  that
 $$|\supp (\widehat f)|\geq \theta (s, \ZZ_p)\theta (t, \ZZ_p)=(p+1-s)(p+1-t),$$
 so that
 $$
\theta(k,G)\geq\min\{(p-s+1)(p-t+1);\;st\leq k,\; 1\leq s,\;t\leq p\} .
$$
As we have already mentioned,   the expression in the right hand side is the same as in the product case. The minimum is attained for $(1, k)$ or $(k,1)$ when $k\leq p$ (resp.
 $(p, j)$ or $(j,p)$ when $jp\leq  k<(j+1)p$, with $j=1, \ldots, p-1$). We have seen by Lemma \ref{H} that there are equality cases. So  the function $\theta $ is equal to this minimum.

It remains to prove that there are no other equality cases. Let us first assume that $k\leq p$. Let $f$ be an equality case. By invariance by translation,
 we can assume that $f(0)\neq 0$ (so that $a$ in \eqref{case1} will be $0$). We define $s$ and $t$ as before. If $t=1$, we conclude directly that the support of $f$ is contained in  $H$ and use Lemma \ref{H}.
 Let us prove that there is no possible equality case $f$ for which $t=k>1$. Assume that there is such a function $f$. Because the support of $f$ has cardinality $k$,
each nonzero $f_a$ is a Dirac mass, so that
\begin{equation}\label{reduction}
f=\sum _{l=1}^k c_l\mathds{1}_{\{a_l+p j_l\}}\qquad \left(
c_\ell\ne 0, \ell=1,\dots,k\leq p \right).
\end{equation}
 The rest of part
(i) obtains directly from the following lemma, because
such a function, whose spectrum has size strictly greater than $p(p-k+1)$, is not an equality case.
\begin{lemma}\label{diagonal}
For $k\geq 2$ let $f$ be a nonzero function that may be written as in \eqref{reduction},
with $a_l$ taking k different values between $0$ and $p-1$ and $j_l$ integers between $0$ and $p-1$. Then, if $k>2$, its spectrum has size at least $p(p-k+2)$. If $k=2$, its spectrum has size at least $p^2-1$.
\end{lemma}
\begin{proof}
  Let us first assume that $k=2$. We assume that $_widehat f$ vanishes at one point, otherwise there is nothing to prove. Without loss of generality (eventually multiplying $f$ by a character), we may assume that $\widehat f$ vanishes at $0$, so that $f$ may be written as $c(\mathds{1}_{\{x\}}- \mathds{1}_{\{y\}})$. We have already seen that such a function does not vanish at any other point, unless $x-y\in H$, which is contrary to the assumption.

  From now on we assume that $k>2$. We have the expression
$$\widehat f(\alpha + p\eta)=\sum_{l=1}^k c_l e^{-2i\pi \frac{(a_l+pj_l)\alpha}{p^2}}e^{-2i\pi \frac{\eta a_l}{p}},$$
where $\alpha$ and $\eta$ may take the values $0, 1, \cdots, p-1$. As a consequence we see that, for fixed $\alpha =0, 1, \cdots, p-1$, the function defined on $\ZZ_p$ by $\eta\mapsto \widehat f(\alpha + p\eta)$  is the Fourier transform of the function $g_\alpha$ on $\ZZ_p$, defined by
$$g_\alpha:= \sum_{l=1}^k c_l e^{-2i\pi \frac{(a_l+pj_l)\alpha}{p^2}}\mathds{1}_{\{a_l\}}.$$
The support of $g_\alpha$ has cardinality $k$. So  $|\supp (\widehat{g_\alpha})|\geq p+1-k$ by
 Theorem \ref{Tao}.We proceed by contradiction and  assume that the cardinality of the support of $\widehat f$ is less than $p(p-k+2)$. Hence there exists $\beta$ such that   $|\supp (\widehat{g_\beta})|= p+1-k$. We call  $z_1, \cdots, z_{k-1}$ the zeros of $\widehat{g_\beta}$. Moreover, one of  the $p-1$ other ones, say $\widehat{g_{\beta'}}$, vanishes at least at one point, say $z_k$. Otherwise, we would have $|\supp (\widehat{f})|=p(p-1)+p-k+1 \geq  p(p-k+2)$, which contradicts our assumption.

   We write $\alpha=\beta'-\beta$ and call $d_l$ the value of $g_\beta$ at $a_l$. By assumption the coefficients $d_l$ are nonzero for $l=1, \cdots, k$. Moreover,
    they are nonzero solutions of the system of $k$ linear homogeneous equations with matrix $(\gamma_{s, l})$, which is obtained as follows.  The $k-1$ first ones express the fact that $\widehat{g_\beta}$ vanishes at $z_1, \cdots, z_{k-1}$,  the last one  the fact $\widehat{g_{\beta'}}(z_k)=0$. Let us write $w=e^{-\frac{2i\pi}{p^2}}$. For $s<k$, we have $\gamma_{s, l}=w^{pa_lz_s}$, while $\gamma_{k, l}=w^{\alpha a_l}w^{p(j_l\alpha +a_lz_k)}$.  A contradiction is obtained if we can prove that the determinant is nonzero. This determinant value is the value at $w$ of a polynomial in one variable $X$ with coefficients in $\ZZ$. If we expand the determinant along the last row,  it can be written as
$$P=\sum_{l=1}^k X^{a_l\alpha}X^{p(j_l\alpha+a_l z_k)} \; (Q_l\circ X^p) ),$$
where $Q_l$'s come from cofactors obtained from the $k-1$ first rows. Up to a multiplicative constant, each $Q_j(w^p)$ is a $(k-1)\times (k-1)$ determinant extracted from the Fourier matrix of $\ZZ_p$, so it does not vanish at $w$ by Chebotarev's Lemma (see \cite{Tao}).

The $a_l$ are all different and the $a_l \alpha$ belong to different congruent classes mod $p$. So, if we reorder the terms,  $P$ can be written as
 $$ P=\sum_{j=0}^{p-1} X^j (P_j\circ X^p),$$
 with $P_j$ that does not vanish at $w$ for $k$ values of $j$.  Now let us assume for a contradiction that $P(w)=0$.
    It follows that the cyclotomic polynomial
    $1+X^p+\dots+X^{(p-1)p}$ divides $P$ in ${\mathbb
    Z}[X]$, that is, with a suitable $Q\in {\mathbb Z}[X]$,
    $P$ can be written as
 $$P= Q(1+X^p+\cdots +X^{(p-1)p}).$$
 We then use the fact that every polynomial can be written in a unique way as a sum $\sum_{j=0}^{p-1} X^{j}(R_j\circ X^p)$ to see that each $P_j\circ X^p$ can also be factorized by the polynomial $1+X^p+\cdots +X^{(p-1)p}$. So it vanishes at $w$, which gives a contradiction and proves that there is no such $f$.
\end{proof}

{\sl Continuation of the proof of Theorem \ref{psquare}.} The theorem has been proved when $1\leq k\leq p$. It is easy to conclude for $k=pj$, with $1\leq j \leq p$, by using Lemma \ref{reciproque}.  We now consider values $k$ such that $pj<k\leq pj+p-1$ and  denote $\ell:=\theta(pj,G)=p-j+1$. Considering Fourier transforms, we are led to prove that there does not exist a function $\phi$ whose support has size $\ell$ and whose Fourier transform has size $k$, with $p(p-\ell +1)< k<p(p-\ell+2)$. We proceed by contradiction. Coming back to Meshulam's proof,
 if we define $s$ and $t$ as before, we know that
 $k=|\supp (\widehat \phi)|\geq (p+1-s)(p+1-t)$ and $st\leq \ell$. The two inequalities $p(p-\ell+1)+p-1\geq (p+1-s)(p+1-t)$ and $st\leq \ell$ imply as before that $t=1$ or $t=\ell$. Recall that $t$ is the number of nonzero $\phi_y$'s. For $t=1$, as in the proof of Lemma \ref{H}, the support of $\widehat \phi$ is a multiple of $p$, which we have excluded.  We use Lemma \ref{diagonal} to see that there is no such function $\phi$ with $t=\ell$.

That concludes the proof of the theorem.
\end{proof}

\noindent {\bf Acknowledgement.} The authors thank the referee for his/her careful reading and for having corrected a mistake in the first version of this paper. The numerous comments of the referee have helped the authors to improve considerably their manuscript.

\end{document}